\documentclass[12pt,reqno]{amsart}
\usepackage{amsmath,amssymb,amsthm,mathtools,calc,verbatim,enumitem,tikz,url,hyperref,mathrsfs,cite,fullpage}
\usepackage{bbm}
\usepackage{stmaryrd}
\usepackage{textcomp}
\usepackage{setspace}
\usepackage{amssymb}
\usepackage{amsthm}
\usepackage{amsmath}
\usepackage{graphicx}
\usepackage{marvosym}
\usepackage{empheq}
\usepackage{latexsym}
\usepackage{fontenc}
\usepackage{color}
\usepackage{hyperref}
\usepackage{cleveref}

\newtheorem{theorem}{Theorem}[section]
\newtheorem{lemma}[theorem]{Lemma}
\newtheorem{corollary}[theorem]{Corollary}
\newtheorem{conjecture}[theorem]{Conjecture}

\newtheorem{proposition}[theorem]{Proposition}

\newtheorem*{claim*}{Claim}

\theoremstyle{definition}
\newtheorem{definition}[theorem]{Definition}

\newtheorem*{qu*}{Question}
\theoremstyle{remark}
\newtheorem{remark}[theorem]{Remark}

\newcommand\N{\mathbb{N}}

\newcommand\E{\operatorname{\mathbb{E}}}

\newcommand\cC{\mathcal{C}}

\renewcommand\Pr{\operatorname{\mathbb{P}}}

\newcommand\eps{\varepsilon}
\renewcommand\leq{\leqslant}
\renewcommand\geq{\geqslant}
\renewcommand\le{\leqslant}
\renewcommand\ge{\geqslant}

	\def\eps{\varepsilon}

	\def\Prob{\mathbb{P}}

	\def\N{\mathbb{N}}

	\def\<{\langle }
	\def\>{\rangle }

\begin{document}

\title{Near Rainbow Hamilton Cycles in Dense Graphs}
\author{Danni Peng and Zhifei Yan}

\address{IMPA, Estrada Dona Castorina 110, Jardim Bot\^anico,
Rio de Janeiro, 22460-320, Brazil}\email{danni.peng@impa.br
\\zhifei.yan@impa.br}

\thanks{}

\begin{abstract}
Finding near-rainbow Hamilton cycles in properly edge-coloured graphs was first studied by Andersen, who proved in 1989 that every proper edge colouring of the complete graph on $n$ vertices contains a Hamilton cycle with at least $n-\sqrt{2n}$ distinct colours. This result was improved to $n-O(\log^2 n)$ by Balogh and Molla in 2019. 
 
In this paper, we consider Anderson's problem for general graphs with a given minimum degree. We prove every globally $n/8$-bounded (i.e. every colour is assigned to at most $n/8$ edges) properly edge-coloured graph $G$ with $\delta(G) \geq (1/2+\varepsilon)n$ contains a Hamilton cycle with $n-o(n)$ distinct colours. Moreover, we show that the constant $1/8$ is best possible.

\end{abstract}
	
	\maketitle
\section{Introduction}

A \emph{proper} edge colouring of a graph $G$ assigns colours to $E(G)$ such that every pair of incident edges receive  different colours. An attractive and longstanding open question is, given an proper edge-coloured complete graph $K_n$, can you always find a rainbow Hamilton cycle, that is, a Hamilton cycle with each edge assigned a different colour? In 1980, Hahn \cite{Hahn} conjectured that every properly edge-coloured $K_n$ with $n\geq 5$ contains a rainbow Hamilton path.\footnote{Observe that if we colour the edges of $K_4$ by assigning three different colours to the three disjoint matchings of size $2$, then there is no rainbow Hamilton path.} However, Maamoun and Meyniel \cite{MM} found a properly edge-coloured $K_n$ without a rainbow Hamilton path for $n=2^k$, which refutes Hahn's conjecture. In 2019, Montgomery, Pokrovskiy and Sudakov \cite{MPS} showed that every properly edge-coloured $K_n$, in which at most $cn$ colours are assigned to 
more than $cn/2$ edges, 
must contain $cn/2$ edge-disjoint rainbow Hamilton cycles, for any constant $0<c<1$ and sufficiently large $n$. In another direction, Anderson \cite{A} proved that every properly edge-colored $K_n$ contains a Hamilton cycle in which at least $n-\sqrt{2n}$ distinct colours appear. Recently, Balogh and Molla \cite{BM} improved the number of distinct colours to $n-O(\log^2 n)$.   

More recently, the study of finding rainbow Hamilton cycles has been extended to general graphs. Coulson and Perarnau \cite{CP} considered this question on  edge-coloured Dirac graphs (i.e. graphs on $n$ vertices with minimum degree at least $n/2$). They proved that every $o(n)$-bounded edge-coloured Dirac graph contains a rainbow Hamilton cycle, where a colouring is $b$-\emph{bounded} if every colour is assigned to at most $b$ edges. Note that Coulson and Perarnau did not require the colouring to be proper; rather, they needed every colour class in the colouring to be strongly controlled, in the sense that it contains at most $o(n)$ edges.

The results mentioned above illustrate two natural ways in which we can constrain an edge colouring in order to find a rainbow Hamilton cycle: \emph{local} and \emph{global} boundedness. Local boundedness refers to bounding the number of edges incident to each vertex that can be assigned the same colour; for example, we can ask the colouring to be proper. In contrast, global boundedness involves controlling the total number of edges of the same colour throughout the entire graph, like the property of being $o(n)$-bounded that was used by Coulson and Perarnau. It is natural to ask whether it is possible to weaken this assumption on the global boundedness of the colouring, if we consider only proper edge-colourings of a Dirac graph. 
Our main result offers a partial answer to this question, by considering near-rainbow Hamilton cycles in graphs $G$ on $n$ vertices with $\delta(G)\geq(1/2+\varepsilon)n$.

\begin{theorem}\label{thm:main}
Let $\varepsilon > 0$ be constant, let $n \in \N$, and let $G$ be a graph on $n$ vertices with minimum degree  $\delta(G) \geq (1/2 + \varepsilon)n$. Then every proper $n/8$-bounded edge-colouring of $G$ contains a Hamilton cycle with $n-o(n)$ distinct colours.
\end{theorem}

Coulson and Perarnau~\cite[Theorem 1.4]{CP} showed for any constant $c > 1/8$, there exists a Dirac graph on $n$ vertices with a $cn$-bounded colouring that contains no rainbow Hamilton cycle. Motivated by their construction, we will show that the constant $1/8$ in Theorem~\ref{thm:main} is optimal, in the following sense.

\begin{proposition}\label{prop:necess}
For each constant $c > 1/8$, there exists $n_0=n_0(c)$ and $\varepsilon = \varepsilon(c) > 0$ such that the following holds for every $n > n_0$. There exists a graph $G$ on $n$ vertices with $\delta(G)\geq (1/2+\varepsilon)n$, and a proper $cn$-bounded edge-colouring of $G$, such that every Hamilton cycle in $G$ contains at most $(1-\varepsilon)n$ distinct colours.
\end{proposition}

If we completely remove the condition on the global boundedness of the colouring, then we cannot even expect to find a near-rainbow Hamilton cycle. Indeed, by Vizing's theorem there exists a proper edge-colouring of any $d$-regular graph $G$ using at most $d+1$ colours, and in such a colouring every Hamilton cycle obviously contains at most $d+1$ distinct colours. In this setting, one can easily deduce the following corollary from Theorem~\ref{thm:main}. 


\begin{corollary}\label{cor:only:proper}
Let $\varepsilon > 0$ be constant, let $n \in \N$, and let $G$ be a graph on $n$ vertices with minimum degree $\delta(G) \geq (1/2 + \varepsilon)n$. Then every proper edge-colouring of $G$ contains a Hamilton cycle with at least $n/4 - o(n)$ distinct colours.
\end{corollary} 



The first step in the proof of \Cref{thm:main} is to construct an `absorber', using a simplified version of a classic construction of R\"{o}dl, Ruci\'{n}ski and Szemer\'{e}di~\cite{RRS}. More precisely, we will find a path $A$ of length $o(n)$ such that for any set $U$ of size $o(v(A))$, there exists a path with vertex set $V(A) \cup U$ and the same endpoints as $A$ (see Lemma~\ref{lem:absorber}). We will also find a `reservoir' set $R$ of size $o(v(A))$ with the property that every pair of vertices of $G$ has many common neighbours in $R$ (see Lemma~\ref{lem:reservoir}). 

The next step is to construct a rainbow path forest $F$ covering almost all of the vertices of $G$. To do so, we will adapt an approach of Montgomery, Pokrovskiy and Sudakov~\cite{MPS}. We will then connect these paths (and $A$) using vertices from $R$, and finally use the property of $A$ to absorb the leftover vertices, noting that the resulting Hamilton cycle has at least as many distinct colours as the rainbow path forest $F$. Since we have no control over the colours of the connecting edges, however, we are unable to find a rainbow Hamilton cycle using this method. 

The remainder of this paper is organized as follows. In \Cref{sec:rainbow forest}, we adapt the method of~\cite{MPS} to construct a rainbow path forest in $G$, in \Cref{sec:absorber} we construct our absorber and our reservoir, and in \Cref{sec:HC} we deduce \Cref{thm:main}. 
Finally, in \Cref{sec:counterexample}, we give a construction of a properly coloured graph with no near-rainbow Hamilton cycles.  
\medskip

\section{Finding a rainbow path forest}\label{sec:rainbow forest}

In this section we will construct a rainbow path forest $F$ covering $n - o(n)$ vertices of a properly coloured graph $G$; that is, a collection of vertex-disjoint paths in which each edge (across all paths) is assigned a distinct colour. To do so, we will adapt an approach of Montgomery, Pokrovskiy and Sudakov~\cite{MPS}, who proved that every proper colouring of the complete graph $K_n$ satisfying certain constraints contains a spanning rainbow forest. Moreover, their result holds for all `typical' graphs, which roughly speaking means graphs whose codegrees can be well controlled. However, their result cannot be applied directly to an arbitrary Dirac graph, or to an arbitrary graph $G$ with minimum degree $\delta(G) \geq (1/2 + \varepsilon)n$. We refine the proof of \cite[Lemma~8.17]{MPS}, showing that a near-spanning rainbow forest still exists in graphs without any constraints on the codegrees. The main aim of the section is to prove the following `rainbow forest lemma' for graphs $G$ satisfying $\delta(G)\geq (1/2+\varepsilon)n$. 

\begin{lemma}[Rainbow forest]\label{lem:key}
Let $0 < \varepsilon \leq 1/2$, and let $\alpha > 0$ be sufficiently small. There exists $n_0=n_0(\varepsilon,\alpha)$ such that the following holds for every $n \ge n_0$, and every graph $G$ on $n$ vertices with $\delta(G) \geq (1/2 + \varepsilon)n$.
In any $\big( n/8 + o(n) \big)$-bounded proper edge colouring of $G$, there exists a rainbow subgraph $F \subset G$ consisting of at most $n^{1-\alpha}$ vertex-disjoint paths, such that every path in $F$ has length at least $n^{\alpha}$ and $v(F) \geq n - 2n^{1-\alpha}$. 
\end{lemma}

Roughly speaking, we will prove \Cref{lem:key} in four steps. First, by a result of Christofides, Kühn and Osthus, we can find an $r$-regular spanning subgraph $G' \subset G$ with $r \approx \delta(G) /2$ (see \Cref{lem:CKO}). Next we randomly partition $V(G')$ into $m = n^\alpha + 1$ equal-sized sets $V_0,\ldots,V_m$ and the colour set $\cC$ into $m$ sets $\cC_1,\ldots,\cC_m$. We will prove that, with high probability, for every adjacent pair $(V_{i-1},V_i)$, the subgraph $B_i$ of $G'[V_{i-1},V_i]$ formed by the edges assigned colours in $\cC_i$ has a certain `near-regularity' property. We will then apply a lemma of Montgomery, Pokrovskiy and Sudakov (see Lemma~\ref{lem:MPS}) to $B_i$, to find a near-perfect rainbow matching in $B_i$ for each $i \in [m]$. Finally, we will greedily construct rainbow paths using these rainbow matchings, to form a large rainbow forest. 

In order to prepare for the formal proof of \Cref{lem:key}, we will need to present the various lemmas mentioned in the outline above. First, in order to find a regular spanning subgraph of $G$ we will use the following lemma, which is a special case of a result of Christofides, Kühn and Osthus~\cite[Theorem 12]{CKO}.

\begin{lemma}\label{lem:CKO}
If\/ $\delta(G) > n/2$, then $G$ has an $r$-regular spanning subgraph $G'\subset G$, where 
$$\lceil\delta(G)/2\rceil\leq r\leq \lceil\delta(G)/2\rceil+1$$
is an even integer.
\end{lemma}

\begin{proof}
This follows from part~$(i)$ of~\cite[Theorem 12]{CKO}, applied with $\ell = 1$. 
\end{proof}

Montgomery, Pokrovskiy and Sudakov \cite{MPS} introduced a coloured version of R\"{o}dl's nibble to construct near-perfect rainbow matchings in a `near-regular' bipartite graph. To apply their method, we need the following definition. 

\begin{definition}\label{reg}
For $\gamma,\delta>0$, a graph $H$ on $n$ vertices is $(\gamma,\delta,n)$-near-regular if
$$(1-\gamma)\delta n \le d_H(v) \le (1+\gamma)\delta n$$
for every $v\in V(H)$.
\end{definition}

As mentioned above, we will randomly partition the vertex set $V(G')$ of the regular graph $G'$ into parts $V_i$ of size $n^{1-\alpha}$. We will use the following lemma from \cite{MPS} to show that, with high probability, each of the induced subgraphs $G'[V_i,V_{i+1}]$ is near-regular. 

\begin{lemma}[Lemma~5.2 of~\cite{MPS}]\label{lem:pickvtxs}
Let $\alpha > 0$ be sufficiently small, let $n \in \N$, and let $n^{-\alpha} \le \gamma,\delta,p,\mu\leq1$ with $p < 1/2$.  Let $H$ be a $(\gamma,\delta,n)$-near-regular graph with a proper $\mu n$-bounded edge-colouring, and let $A,B\subset V(H)$ be disjoint subsets of order $pn$, chosen uniformly at random from all such pairs of subsets. Then $H[A,B]$ is a  $(2\gamma,\delta,2pn)$-near-regular bipartite graph with a $\big( (1+\gamma)2\mu p^2 n \big)$-bounded colouring with probability $1-o(n^{-1})$.
\end{lemma}

We aim to find a near-perfect rainbow matching in $G'[V_{i-1},V_{i}]$ for each $i \in [m]$, ensuring moreover that each edge (across all these matchings) has a distinct colour. To achieve this, we randomly partition the colour set $\cC$ into $m$ parts $\cC_i$, and focus on the subgraph $B_i \subset G'[V_{i-1}, V_{i}]$ consisting of the edges assigned a colour in $\cC_i$. The following lemma from \cite{MPS} guarantees that there exists a partition such that each $B_i$ is near-regular. 

\begin{lemma}[Lemma~5.3 of~\cite{MPS}]\label{lem:pickcors}
Let $\alpha > 0$ be sufficiently small, let $n \in \N$, and let $n^{-\alpha} \le \gamma,\delta,p,\mu \leq 1$. Let $H$ be a balanced bipartite $(\gamma,\delta,n)$-near-regular graph with a proper $\mu n$-bounded edge-colouring, and let $H_1\subset H$ be the random subgraph formed by the edges whose colour lies in $S$, where each colour is included in $S$ independently with probability $p$. Then $H_1$ is $(2\gamma,p\delta,n)$-near-regular with probability $1-o(n^{-1})$.
\end{lemma}

We will use the following lemma from \cite{MPS} to find a rainbow near-perfect matching in the graph $B_i$ for each $i \in [m]$.


\begin{lemma}[Lemma~8.8 of~\cite{MPS}]\label{lem:MPS}
Let $\alpha > 0$ be sufficiently small, let $n \in \N$ be sufficiently large, and let $n^{-2\sqrt{\alpha}} \le \gamma \leq n^{-\sqrt{\alpha}}$ and $n^{-2\alpha} \le \delta, q \leq 1$. Let $H$ be a balanced bipartite $(\gamma,\delta,n)$-near-regular graph on $n$ vertices with a proper $(1-q)\delta n/2$-bounded edge-colouring. Then there exists a rainbow matching $M \subset H$ of size at least
$$e(M) \geq (1-q)n/2.$$
\end{lemma}

We will construct a rainbow path forest $F$ greedily using these rainbow matchings. 

\begin{lemma}\label{lem:mat->F}
Let $m,n,t \in \N$, let $G$ be a graph with $(m+1)n$ vertices, and let $V(G) = V_0 \cup \cdots \cup V_m$ be a partition with $|V_0| = \cdots = |V_m| = n$. Suppose that 
 $G[V_{i-1}, V_{i}]$ contains a matching $M_i$ with $e(M_i) \geq n - t$ for each $i \in [m]$. Then there exists $H \subset G$ consisting of $n - tm$ vertex-disjoint paths of length $m$. 
\end{lemma}

\begin{proof}
Let $0 \leq k \leq m-1$, and suppose we have already built a path forest $H_k \subset G\big[ V_0 \cup \cdots \cup V_k \big]$ consisting of $n - tk$ vertex-disjoint paths of length $k$. 
At most $t$ of the paths in $H_k$ cannot be extended by adding an edge of $M_{k+1}$, and we therefore obtain a path forest $H_{k+1} \subset G\big[ V_0 \cup \cdots \cup V_{k+1} \big]$ consisting of $n - t(k+1)$ vertex-disjoint paths of length $k+1$. Since this holds for every $0 \leq k \leq m-1$, we obtain a path forest $H = H_m$ as required.
\end{proof}


We are finally ready to put the pieces together and prove \Cref{lem:key}. 

\begin{proof}[Proof of \Cref{lem:key}]
Fix $\alpha > 0$ sufficiently small and $0 < \varepsilon \leq 1/2$, and let $n \in \N$ be sufficiently large. Suppose that we are given a graph $G$ on $n$ vertices with $\delta(G) \geq (1/2 + \varepsilon)n$, and a $n/8$-bounded proper edge colouring of $G$. By \Cref{lem:CKO}, there exists an $r$-regular spanning subgraph $G'\subset G$ with $r \geq \delta(G)/2$. Note that $G'$ is $(\gamma,\delta,n)$-near-regular for $\gamma = n^{-\sqrt{\alpha}}$ and $\delta = r/n$, and the colouring of $G'$ is proper and $n/8$-bounded, since $G' \subset G$.

Now, set $m = \lfloor n^\alpha \rfloor - 1$, and let $V(G') = U \cup V_0 \cup \cdots \cup V_m$ be a random partition of the vertices of $G'$ with
$$|U| < |V_0| = \cdots = |V_m| = n' := \bigg\lfloor \frac{n}{m+1} \bigg\rfloor.$$
By \Cref{lem:pickvtxs} and the union bound, with high probability the coloured graph $G[V_{i-1},V_i]$ has the following two properties for each $i \in [m]$: it is $(2\gamma,\delta,2n')$-near-regular, and the colouring is $(1 + \gamma) n'/4m$-bounded. 

We next randomly choose a partition $\cC = \cC_1 \cup \cdots \cup \cC_m$ of the set of colours into $m$ parts, by placing each colour independently and uniformly into one of the parts. For each $i \in [m]$, let $B_i$ be the subgraph of $G'[V_i,V_{i+1}]$ formed by the edges whose colour is placed in $\cC_i$. Applying \Cref{lem:pickcors} with $H = G'[V_i,V_{i+1}]$, $p = 1/m$ and $S = \cC_i$, and taking a union bound over $i \in [m]$, it follows that with high probability every $B_i$ is $(4\gamma,\delta/m,2n')$-near-regular. Moreover, the edge-colouring of $B_i$ is proper and $(1 - \gamma)\delta n' / m$-bounded, 
since $B_i \subset G'[V_i,V_{i+1}]$ and $\delta = r/n \ge 1/4 + \eps/2$.

Let's fix partitions of $V(G')$ and $\cC$ such that $B_i$ satisfies these conditions. By \Cref{lem:MPS}, applied with $q = n^{-2\alpha}$, it follows that there exists a rainbow matching $M_i \subset B_i$ such that
\begin{align*}
e(M_i) \geq (1-q)n'.    
\end{align*}
Finally, by \Cref{lem:mat->F}, it follows that there exists a rainbow path forest $F\subset G'$ consisting of $|V_i| - qmn'$ vertex-disjoint paths of length $m$. Thus 
\begin{align*}
v(F) \geq n - n^{1 - \alpha} + qm n \geq n - 2n^{1 - \alpha}, 
\end{align*}
as required, since $q \le n^{-2\alpha}$.
\end{proof}

\begin{remark}
In the proof of \Cref{lem:key}, we actually only needed the edge colouring of $G$ to be proper and $\big( \big(1/4 - o(1) \big)\,\delta(G) \big)$-bounded.
\end{remark} 

\smallskip

\section{Finding an absorber and a reservoir}\label{sec:absorber}

In this section, we will construct an absorber and a reservoir in a graph $G$ with minimum degree $\delta(G) \geq (1/2 + \varepsilon)n$. We begin with the following lemma, which provides our absorber; both the statement and proof are based on~\cite[Lemma 2.3]{RRS}.

\begin{lemma}[Absorber]\label{lem:absorber}
For $0 < \varepsilon < 1/2$ and $0 < c < 2^{-10}\varepsilon$, there exists $C = C(c,\varepsilon)$ and $n_0 = n_0(c,\varepsilon)$ such that the following holds for every $n \geq n_0$ and $(\log n)^2 \leq m \leq n/C$. If $G$ is a graph on $n$ vertices with $\delta(G) \geq (1/2 + \varepsilon)n$, then 
there exists a path $A\subset G$ with
$$v(A)\leq Cm,$$
such that for every set $U\subset V(G)\setminus V(A)$ with size $|U|\leq cm$, there is a path $A_{U}\subset G$ which has the same endpoints as $A$ and vertex set $V(A_U) = V(A) \cup U$.
\end{lemma}

The idea of the proof is to find a small matching $M \subset G$ which nevertheless contains \emph{many} edges of $G[N(v)]$ for \emph{every} vertex $v \in V(G)$. The absorber $A$ will be a path formed by adding paths of length two between consecutive pairs of edges of $M$. The matching $M$ is given by the following lemma.

\begin{lemma}\label{lem:cF}
For $0<\varepsilon<1/2$ and $0<c<2^{-10}\varepsilon$, there exists $C = C(c,\varepsilon)$ and $n_0 = n_0(c,\varepsilon)$ such that the following holds for every $n \geq n_0$ and $(\log n)^2 \leq m \leq \eps n$. If $G$ is a graph on $n$ vertices with $\delta(G) \geq (1/2 + \varepsilon)n$, then there exists a matching $M\subset G$ with $e(M)\leq Cm$ such that
$$e\big(M\cap G[N(v)]\big)>cm$$
holds for every vertex $v\in V(G)$.
\end{lemma}

\begin{proof}
First, note that for every $v\in V(G)$, the number of edges in $G[N(v)]$ is at least
$$e(G[N(v)])\geq\varepsilon n^2/2.$$
Indeed, since $\delta(G) \geq (1/2 + \varepsilon)n$, it follows that $|N(u) \cap N(v)| \geq 2\varepsilon n$ for every $u,v \in V(G)$, and therefore $e(G[N(v)]) \geq \delta(G) \cdot \varepsilon n \geq \varepsilon n^2/2$, as claimed.

Set $C = 8c/\varepsilon$, and let $G^*\sim G(n,p)$ be a random graph on $V(G)$ with $p = Cm/n^2 \in (0,1)$. Since $m \geq (\log n)^2$, by applying Chernoff's inequality and taking a union bound we deduce that with high probability the following holds: the number of edges in $G^*$ is at most $Cm$,
and for every $v\in V(G)$, the number of edges in $G[N(v)]\cap G^*$ is at least 
$$e\big(G[N(v)]\cap G^*\big)\geq e(G[N(v)])\cdot p\cdot 1/2 \geq \frac{Cm}{n^2} \cdot \frac{\eps n^2}{4} = 2cm.$$
For any graph $H$, let $E'(H)\subset E(H)$ be the set of edges that are \emph{not isolated}, meaning that they share an endpoint with some other edge. Observe that $|E'(G)|$ is at most the number of walks of length $2$ in $G$, so $\E[ |E'(G^*)| ] \leq p^2 n^3 \leq C^2m^2 / n$, 
and therefore
$$\Prob\big( |E'(G^*)| \geq cm \big) \leq \frac{\E[ |E'(G^*)| ]}{cm} \leq \frac{C^2m^2}{cmn} \leq\frac{2^6 c}{\varepsilon} \leq 2^{-4}$$
by Markov's inequality, and since $m \le \eps n$ and $c < 2^{-10}\varepsilon$. 
Hence there exists a graph $G_0$ with $V(G_0)=V(G)$, such that $e(G_0)\leq Cm$, $|E'(G_0)|<cm$ and
$e\big(G[N(v)]\cap G_0\big)\geq 2cm$ for every $v\in V(G)$. 

Finally, let $M$ be the largest matching in $G_0\cap G$. Observe that $e(M) \leq e(G_0) \leq Cm$, and that for every $v\in V(G)$, the number of edges in $G[N(v)]\cap M$ is at least
$$e\big(G[N(v)]\cap M\big)\geq  e\big(G[N(v)] \cap G_0\big) -|E'(G_0)|> cm,$$
as required.
\end{proof}

We are now ready to construct our absorber, using the matching given by \Cref{lem:cF}.

\begin{proof}[Proof of \Cref{lem:absorber}]
Let $C_0 = C_0(\eps,c) > 0$ be the constant given by \Cref{lem:cF}, and set $C = 2 C_0 / \eps$, so $C_0 m \le \varepsilon n/2$. By \Cref{lem:cF}, we obtain a matching $M \subset G$ with edge set
$$E(M) = \big\{ x_1y_1, x_2y_2, \ldots, x_ty_t \big\}$$
for some $t \leq C_0 m \leq \varepsilon n / 2$, such that
\begin{align}\label{eq:qq1}
e\big(M\cap G[N(v)]\big) > cm
\end{align}
for every vertex $v\in V(G)$. We now construct the path $A$ by connecting consecutive edges of $M$ using $t-1$ additional vertices. To be precise, suppose that for some $1 \leq k \leq t-1$
we have already found a path $A_k$ with 
$$E(A_k) = \big\{ x_1y_1, y_1z_1, z_1x_2, x_2y_2, \ldots, y_{k-1}z_{k-1}, z_{k-1}x_k, x_{k}y_{k} \big\}$$
where $z_i\notin V(M)$ and $z_i\neq z_j$ for $i\neq j$. Since $|N(y_k)\cap N(x_{k+1})|\geq 2\varepsilon n > 3t$, there exists 
$$z_k \in N(y_k) \cap N(x_{k+1}) \setminus \big( V(M) \cup V(A_k) \big)$$ 
and we can therefore continue this process until we obtain a path $A_t$ containing all of the edges of $M$. Set $A = A_t$,  and observe that $|A| = 3t+1 \leq Cm$.

Finally, we need to show that $A$ has the required `absorption' property. To see this, let $U = \{ u_1,u_2,...,u_s \} \subset V(G) \setminus V(A)$ with $s \leq cm$, and construct the path $A_U$ as follows. 
Let $0 \leq \ell < s$, and assume that we have already found a path $A'_\ell$ with the same endpoints as $A$, containing $t - \ell$ edges of $M$, and with 
$$V(A'_\ell) = V(A) \cup \big\{ u_1, \ldots, u_\ell \big\}.$$ 
We construct a path $A'_{\ell+1}$ by replacing an edge $x_iy_i \in E(M) \cap E(A'_\ell)$ by a path of length $2$ with edge set $\{ x_i u_{\ell+1}, u_{\ell+1} y_i \}$. We can do so because
$$e\big( M \cap A'_\ell \cap G[N(u_{\ell+1})] \big) \geq e\big( M \cap G[N(u_{\ell+1})] \big) - \ell > cm - s \ge 0,$$
by~\eqref{eq:qq1}. We may therefore continue this process until we obtain a path $A'_s$ with the same endpoints as $A$ and containing all of the vertices of $U$, as required.
\end{proof}

The second goal of this section is to prove the following `reservoir' lemma, which is based on~\cite[Lemma 2.7]{RRS}.

\begin{lemma}[Reservoir]\label{lem:reservoir}
For every $0 < c,\varepsilon < 1/2$ and $C > 1$, there exists $n_0 = n_0(c,\varepsilon,C)$ such that the following holds for every $n \geq n_0$ and $(\log n)^2 \leq m \leq \eps n/2C$. If $G$ is a graph on $n$ vertices with $\delta(G) \geq (1/2 + \varepsilon)n$, and $W \subset V(G)$ is a set with $|W| \leq Cm$, then there exists a set $R \subset V(G)\setminus W$ with $|R| = cm$ such that
$$|N(x)\cap N(y)\cap R| \geq \varepsilon|R|$$
for every pair of vertices $x,y\in V(G)$.
\end{lemma}

\begin{proof}
Let $R^*$ be a random set of size $cm$, chosen uniformly from all $cm$-sets in $V(G) \setminus W$. For each $x,y \in V(G)$, let $X_{x,y} = | N(x)\cap N(y)\cap R^* |$, and note that
\begin{align*}
\E\big[ X_{x,y} \big] \ge \frac{|N(x)\cap N(y)|-|W|}{n-|W|} \cdot |R^*| \geq \frac{2\varepsilon n - Cm}{n - Cm} \cdot cm \geq \frac{3\varepsilon cm}{2},
\end{align*}
where the final inequality holds because $Cm < \varepsilon n/2$. By Hoeffding's inequality~\cite{H}, it follows that
\begin{align*}
\Prob\big( X_{x,y} < \varepsilon |R^*| \big) \leq \Prob\big( X_{x,y} < 2\E[X_{x,y}]/3 \big) \leq \exp\big( -\Theta(m)\big) \le n^{-3}, 
\end{align*}
since $n \ge n_0$ and $m \ge (\log n)^2$. By the union bound over pairs $x,y \in V(G)$, it follows that there exists a set $R$ of size $cm$ such that $X_{x,y} \ge \varepsilon |R^*|$ for every $x,y \in V(G)$, as required.
\end{proof}

\smallskip

\section{Finding a near-rainbow Hamilton cycle}\label{sec:HC}

Using the rainbow forest given by  \Cref{lem:key}, the absorber given by \Cref{lem:absorber} and the reservoir given by \Cref{lem:reservoir}, we can now deduce Theorems~\ref{thm:main}. 
In fact, we will prove the following slight strengthening of~\Cref{thm:main}.


\begin{theorem}\label{thm:main1'}
Given $\beta > 0$ be sufficiently small and $0 < \varepsilon \leq 1/2$, there exist $b = b(\varepsilon,\beta)$ and $n_0 = n_0(\varepsilon,\beta)$ such that
the following holds. For every $n \geq n_0$, and every globally $n/8$-bounded proper edge-coloured graph $G$ on $n$ vertices with minimum degree $\delta(G) \geq (1/2 + \varepsilon)n$, there exists a Hamilton cycle in $G$ with at least $n-bn^{1-\beta}$ distinct colours.
\end{theorem}

\begin{proof}
First let's construct an absorber $A$ and a reservoir $R$ in $G$. By \Cref{lem:absorber}, applied with $c = 2^{-11}\varepsilon$ and $m = n^{1-\beta}$, there exists a constant $C = C(\eps) > 0$ and a path $A \subset G$ with 
$$v(A) \le C n^{1-\beta},$$ 
such that for any set $U\subset V(G)\setminus V(A)$ with $|U|\leq cn^{1-\beta}$, there is a path $A_{U}\subset G$ with the same endpoints as $A$ and with $V(A_U) = V(A) \cup U$. In other words, 
$A$ can absorb any set $U$ with $|U|\leq cn^{1-\beta}$. Moreover, by \Cref{lem:reservoir} applied with $W = V(A)$, there exists a set $R \subset V(G) \setminus V(A)$ of size $cn^{1-\beta}/2$ 
such that 
\begin{equation}\label{eq:R:common:nbrs}
|N(v)\cap N(u)\cap R| \geq \varepsilon|R|\geq \varepsilon cn^{1-\beta}/2
\end{equation}
for every $x,y\in V(G)$.

Next let's construct a near-spanning rainbow forest in $G_0 = G - \big( V(A) \cup R \big)$. Note that
\begin{align*}
v(G_0) = n - v(A) - |R| \qquad \textup{and} \qquad \delta(G_0) \geq (1/2 + \varepsilon/2) n,
\end{align*}
since $\beta > 0$ and $n$ is sufficiently large. Moreover, the colouring of $G_0$ is  globally $n/8$-bounded, since $G_0 \subset G$. Since $v(G_0) = n - o(n)$, we may therefore apply \Cref{lem:key} to $G_0$ with $\alpha = 2\beta$, and obtain a rainbow path-forest $F \subset G_0$ consisting of at most $n^{1-\alpha} = n^{1 - 2\beta}$ vertex-disjoint paths and with
$$v(F) + v(A) + |R| \ge n - 2n^{1 - \alpha} = n - o(n^{1-\beta}).$$

We will now use the vertices of $R$ to connect the vertices of $A$ and the paths in $F$ in order to construct an almost-rainbow Hamilton cycle. Let $x_0$ and $y_0$ be the endpoints of $A$, and let $F = P_1 \cup \cdots \cup P_T$, where $P_i$ is a path with endpoints $\{x_i,y_i\}$ and $T \leq n^{1 - 2\beta}$. For each $0 \leq i \leq T$, the vertices $y_i$ and $x_{i+1}$ (where $x_{T+1}=x_0$) have at least 
$$|N(y_i) \cap N(x_{i+1}) \cap R| \geq \varepsilon c n^{1-\beta}/2 > 4n^{1 - 2\beta}$$
common neighbors in $R$, by~\eqref{eq:R:common:nbrs}, and since $\beta > 0$. Hence we can greedily choose vertices $z_i \in N(y_i)\cap N(x_{i+1})\cap R\setminus\{z_0,...,z_{i-1}\}$ for each $0\leq i\leq T$, since $T+1 \leq n^{1 - 2\beta}$. Adding the edges $\big\{y_iz_i, z_ix_{i+1} : 0 \leq i \leq T \big\}$ to the graph $A \cup F$, we obtain a cycle $H$ with
$$v(H) \geq v(F) + v(A) \geq n - |R| - o(n^{1-\beta})$$
whose edges have at least $e(F) \ge n - bn^{\beta}$ distinct colours, where $b = C + c$. 

Finally, let $U = V(G) \setminus V(H)$, and observe that since
$|U| \leq |R| + o(n^{1-\beta}) \le cn^{1-\beta}$, there exists a path $A_{U} \subset G$ with the same endpoints as $A$ and with $V(A_U) = V(A) \cup U$. Replacing the path $A$ by $A_U$ in $H$, we obtain a near-rainbow Hamilton cycle, as required.
\end{proof}


It follows easily from \Cref{thm:main} that, even without any global boundedness condition, every proper edge-colouring of $G$ contains a Hamilton cycle with $n/4 - o(n)$ distinct colours. 

\begin{proof}[Proof of Corollary~\ref{cor:only:proper}]
Note first that any proper edge-colouring $c$ of a graph $G$ on $n$ vertices is $n/2$-bounded. We can therefore construct a proper $n/8$-bounded edge-colouring of $G$ from $c$ by replacing each colour in $c$ by four colours. Applying \Cref{thm:main} to this new colouring $c'$, we obtain a Hamilton cycle in $G$ with $n - o(n)$ colours in $c'$, and therefore with at least $n/4 - o(n)$ distinct colours in $c$, as required. 
\end{proof}

\medskip

\section{Globally-bounded colourings without near-rainbow Hamilton cycles}\label{sec:counterexample}

In this section we will show that the constant 1/8 in \Cref{thm:main} is best possible. This is an immediate consequence of the following more general theorem.

\begin{theorem}\label{thm:main2}
Let $n,m,s,t \in \N$ satisfy $m,t\leq2^{-20}n$, $s \le n/8$, and
\begin{align}\label{eq:nmkt}
2^{-10}s\geq\max\big\{ \sqrt{mn},\, \sqrt{tn},\, (n\log n)^{2/3}\big\}.
\end{align}
Then there exists a coloured graph $G$ on $n$ vertices such that
\begin{enumerate}
\item[$(i)$] $\delta(G) \ge n/2 + m;$

\item[$(ii)$] the colouring of $G$ is proper and $(n/8+s)$-bounded;

 \item[$(iii)$] every Hamilton cycle in $G$ has at most $n-t$ distinct colours.
\end{enumerate}
\end{theorem}
\medskip

Before proving \Cref{thm:main2}, let us quickly note that it implies \Cref{prop:necess}.

\begin{proof}[Proof of \Cref{prop:necess}]
Let $c'=\min\{c,1/4\}$. Observe that if 
\begin{align*}
\varepsilon =(c'-1/8)^2\cdot2^{-20},\qquad s=(c'-1/8)n,  \qquad m=t=\varepsilon n,
\end{align*}
then~\eqref{eq:nmkt} holds, since $1/8 < c' \le 1/4$ and $n > n_0(c)$ is sufficiently large. By \Cref{thm:main2}, it follows that there exists a graph $G$ on $n$ vertices with $\delta(G) \ge (1/2+ \eps)n$, and a proper and globally $c'n$-bounded colouring of $G$ such that every Hamilton cycle in $G$ has at most $(1-\eps)n$ distinct colours, as required.
\end{proof}

In the proof of \Cref{thm:main2}, we will need the following simple lemma.

\begin{lemma}\label{lem:GB}
Let $n,d,k,\ell \in \N$, 
with $\ell \le n/2$ and $d \le 2^{-5} n$. If 
\begin{equation}\label{eq:GB:condition}
\frac{2k\ell}{n} > d + 2^5 d^2/n + 4\sqrt{d} \log n, 
\end{equation} 
then there exists a graph $G$ with $n$ vertices and $\delta(G) \ge d$, and a proper edge-colouring of $G$ using at most $k$ colours with each colour being used on at most $\ell$ edges.  
\end{lemma}


\begin{proof}
Let $M_1,...,M_k$ be random matchings of size $\ell$, chosen uniformly and independently at random in $K_n$, let $H$ be the multigraph formed by taking the union of these matchings, and let $G$ be the graph formed by choosing one edge for each multiple edge of $H$. Define a colouring of $G$ by giving colour $i$ to each edge of $M_i$. 
By construction, this is a proper colouring using at most $k$ colours, and each colour is used on at most $\ell$ edges. 

We claim that $\delta(G) \ge d$ with positive probability. To see this, observe first that
$$d_H(v) \sim \textup{Bin}(k,2\ell/n)$$
for each $v \in V(H)$, and that therefore
$$\Pr\big( d_H(v) < d + h \big) \le e^{-h^2/4d} \leq \frac{1}{n^2},$$ 
where $h = 2^4d^2/n + 2\sqrt{d} \log n$, by Chernoff's inequality and~\eqref{eq:GB:condition}. 



Now, let $A_v$ denote the event $\{ d_H(v) \ge d + h \}$, and partition $A_v$ according to the first $d + h$ matchings (in the order $(M_1,...,M_k)$) that use an edge incident to $v$. Let $B_v$ be the `bad' event $\{ d_G(v) < d \}$, and observe that if $A_v \cap B_v$ occurs, then there exists a set of $h$ of these $d + h$ matchings that contains an edge incident to $v$ that was already used by one of the earlier matchings. Note also that, conditioned on the event that a matching uses an edge incident to $v$, it uses one of the (at most $d$) previously-chosen edges with probability at most $d/(n-1)$. Taking a union bound over the choice of $d+h$ matchings and the choice of $h$ matchings that repeat an edge, it follows that 
\begin{align*}
\Prob\big( A_v \cap B_v \big) & \leq \binom{d+h}{h} \cdot \bigg( \frac{d}{n-1} \bigg)^{h} \leq \bigg( \frac{ed(d+h)}{h(n-1)} \bigg)^{h} \le e^{-h} \le \frac{1}{n^2},
\end{align*}
since $h = 2^4d^2/n + 2\sqrt{d} \log n$. 
Finally by taking a union bound over $v \in V(G)$, it follows that $\delta(G) \ge d$ with positive probability, as claimed.
\end{proof}

We are now ready to prove \Cref{thm:main2}.

\begin{proof}[Proof of \Cref{thm:main2}]
Given $n,m,s,t \in \N$ with $m,t \leq 2^{-20}n$ and satisfying~\eqref{eq:nmkt}, set $q=2^{-10}s$ and consider the following construction. First, let $H$ be graph given by \Cref{lem:GB}, applied with
$$v(H) = n/2+q, \quad d = q + m, \quad \ell = n/8+s \quad \text{and} \quad k = 2q - t.$$
To see that the conditions of the lemma are satisfied, note first that $2^{10} (m+t) \le s \le n/8$, by~\eqref{eq:nmkt}, and therefore $\ell \le v(H)/2$ and $d \le 2^{-5} v(H)$, and moreover
$$\frac{2k\ell}{v(H)} = \frac{2(2q - t)(n/8+s)}{n/2 + q} > q + \frac{8qs}{n} - t - \frac{4q^2}{n}.$$
Since $qs \ge (m+t)n$ and $q^{1/2} s \ge 4 n \log n$, by~\eqref{eq:nmkt}, it follows that
$$\frac{2k\ell}{v(H)} > q + m + \frac{2^{8}q^2}{n} + 8\sqrt{q} \log n > d + \frac{2^6 d^2}{n} + 4\sqrt{d} \log n,$$
as required. Thus, by \Cref{lem:GB}, there exists a properly edge-coloured graph $H$ with $n/2 + q$ vertices and minimum degree $\delta(H) \ge q+m$, using at most $2q-t$ colours, and with each colour being used on at most $n/8 + s$ edges. 



Now, let $G$ be the graph on $n$ vertices obtained from $H$ by adding $n/2 - q$ vertices, each with neighbourhood $V(H)$, and giving each of these new edges a (new) distinct colour. Note that $\delta(G) \ge n/2 + m$, since $q \ge m$, and that the colouring of $G$ is proper and $(n/8+s)$-bounded, since $H$ has this property and each new edge is given a different colour. Finally, note that each Hamilton cycle in $G$ uses at least $2q$ edges of $H$, and therefore contains at most $n - t$ distinct colours, since the colouring of $H$ uses only $2q - t$ different colours. 
\end{proof}





\Cref{thm:main2} 
also implies that one 
can't expect to find a rainbow Hamilton cycle in an arbitrary properly coloured Dirac graph if the colouring is not $\big(n/8 + o(n)\big)$-bounded.

\begin{corollary}\label{prop:123}
For all sufficiently large $n \in \N$, there exists a graph $G$ on $n$ vertices with $\delta(G) = n/2 + 1$ and a $\big(n/8 + 2^{10}(n\log n)^{2/3}\big)$-bounded proper edge-colouring of $G$ that does not contain a rainbow Hamilton cycle.    
\end{corollary}

\begin{proof}
Set $s=2^{10}(n\log n)^{2/3}$, $m=1$ and $t=1$, and apply \Cref{thm:main2}. 
\end{proof}

Nevertheless, we suspect that being $n/8$-bounded \emph{is} sufficient to guarantee that a proper edge-colouring of a Dirac graph contains a rainbow Hamilton cycle. To be precise, we make the following conjecture.

\begin{conjecture}\label{conj}
There exists a rainbow Hamilton cycle in any proper $n/8$-bounded edge-colouring of a Dirac graph on $n$ vertices.
\end{conjecture}





We also make the following conjecture for proper edge-colourings with no global boundedness condition. 

\begin{conjecture}\label{conj:only:proper}
Every proper edge-colouring of a Dirac graph on $n$ vertices contains a Hamilton cycle with at least $n/2 - o(n)$ distinct colours.   
\end{conjecture}

Note that Conjecture~\ref{conj:only:proper} would strengthen Corollary~\ref{cor:only:proper}, which gives a bound of $n/4 - o(n)$, and would be close to best possible, since there exist Dirac graphs with chromatic index $n/2 + 1$. It is even possible that the following stronger version of Conjecture~\ref{conj:only:proper} holds: every proper edge-colouring of a Dirac graph contains a Hamilton cycle with $\delta(G)$ distinct colours.






Finally, let us quickly observe that the conclusion of Conjecture~\ref{conj:only:proper} holds if $\delta(G) \ge 7n/8$. This follows from our method, using the following result of Alon, Pokrovskiy and Sudakov~\cite{APS} in place of Lemma~\ref{lem:key}. 


\begin{lemma}[Lemma 3.1 of \cite{APS}]\label{lem:APS}
Let $\delta \ge \gamma > 0$ and $n \in \N$, with $3\gamma\delta - \gamma^2/2 > n^{-1}$, and let $G$ be a properly
coloured graph on $n$ vertices and $\delta(G) \geq (1 - \delta)n$. Then $G$ contains a rainbow path forest $F$ with at most $\gamma n$ vertex-disjoint paths and
$e(F) \geq (1 - 4\delta)n$.
\end{lemma}

Applying Lemma~\ref{lem:APS} with $\delta = 1/8$ and $\gamma = o(n)$, and using the absorber and reservoir given by Lemmas~\ref{lem:absorber} and~\ref{lem:reservoir} to extend the rainbow path forest to a Hamilton cycle,  it follows that every proper edge-colouring of a graph $G$ on $n$ vertices with $\delta(G) \ge 7n/8$ contains a Hamilton cycle with at least $n/2$ distinct colours. 


\section*{Acknowledgements}

We are very grateful to our supervisor Rob Morris for many helpful discussions on the proofs and presentation of this paper.


\medskip
\bibliographystyle{alpha}

\end{document}